\documentclass[12pt,a4paper,leqno]{amsart}

\usepackage[latin1]{inputenc}
\usepackage[T1]{fontenc}
\usepackage{amsfonts}
\usepackage{amsmath}
\usepackage{amssymb}
\usepackage{eurosym}
\usepackage{mathrsfs}
\usepackage{palatino}
\usepackage{color}
\usepackage{esint}
\usepackage{url}
\usepackage{verbatim}

\usepackage{enumerate}

\newcommand{\R}{\mathbb{R}}
\newcommand{\C}{\mathbb{C}}

\newcommand{\Z}{\mathbb{Z}}

\newcommand{\calM}{\mathcal{M}}
\newcommand{\calS}{\mathcal{S}}

\newcommand{\bla}{\big \langle}
\newcommand{\bra}{\big \rangle}

\newcommand{\vare}{\varepsilon}

\numberwithin{equation}{section}

\newcommand{\dist}[0]{\operatorname{dist}}

\newcommand{\BMO}[0]{\operatorname{BMO}}

\newcommand{\good}[0]{\operatorname{good}}

\newcommand{\calD}[0]{\mathcal{D}}

\swapnumbers
\theoremstyle{plain}
\newtheorem{thm}[equation]{Theorem}
\newtheorem{lem}[equation]{Lemma}
\newtheorem{prop}[equation]{Proposition}
\newtheorem{cor}[equation]{Corollary}

\theoremstyle{definition}
\newtheorem{defn}[equation]{Definition}

\theoremstyle{remark}

\title[Bilinear representation theorem]{Bilinear representation theorem}

\author{Kangwei Li}
\address[K.L.]{BCAM (Basque Center for Applied Mathematics), Alameda de Mazarredo 14, 48009 Bilbao, Spain} 
\email{kli@bcamath.org}

\author{Henri Martikainen}
\address[H.M.]{Department of Mathematics and Statistics, University of Helsinki, P.O.B. 68, FI-00014 University of Helsinki, Finland}
\email{henri.martikainen@helsinki.fi}

\author{Yumeng Ou}
\address[Y.O.]{Department of Mathematics, Massachusetts Institute of Technology, 77 Massachusetts Avenue, Cambridge, MA 02139, USA}
\email{yumengou@mit.edu}

\author{Emil Vuorinen}
\address[E.V.]{Department of Mathematics and Statistics, University of Helsinki, P.O.B. 68, FI-00014 University of Helsinki, Finland}
\email{emil.vuorinen@helsinki.fi}

\makeatletter
\@namedef{subjclassname@2010}{%
  \textup{2010} Mathematics Subject Classification}
\makeatother


\begin{document}

\begin{abstract}
We represent a general bilinear Calder\'on--Zygmund operator as a sum of simple dyadic operators. The appearing dyadic operators
also admit a simple proof of a sparse bound. In particular, the representation implies a so called sparse $T1$ theorem for bilinear singular integrals.
\end{abstract}

\maketitle

\section{Introduction}
In this paper we show the exact dyadic structure behind \emph{bilinear} Calder\'on--Zygmund operators by representing them using simple dyadic operators, namely
some cancellative bilinear shifts and bilinear paraproducts. In the linear case Petermichl \cite{Pe} first represented the Hilbert transform in this way, and later
Hyt\"onen \cite{Hy} proved a representation theorem for all linear Calder\'on--Zygmund operators.

The representation theorems were originally motivated by the sharp weighted $A_p$ theory, but certainly also have other value and intrinsic interest.
For example, a representation theorem holds also in the bi-parameter setting as shown by one of us \cite{Ma1} (the multi-parameter extension of this is also by one of us \cite{Yo}), and in this context the representation has proved to be very useful e.g. in connection
with bi-parameter commutators, see \cite{HPW} and \cite{OPS}.

Outside the multi-parameter context it is true that sparse domination results yield sharp weighted bounds, and that sparse domination can also be proved directly (without going through a representation).
Such proofs usually start from the unweighted boundedness assumption, then conclude some weak type estimates, and then finally go about proving the sparse domination. However, we think
that the idea of a so called sparse $T1$, as coined by Lacey--Mena \cite{LM}, is extremely practical.  This amounts to concluding a sparse bound directly from the $T1$ assumptions (by modifying
the probabilistic $T1$ proof), and then noting
that the sparse bound implies all the standard boundedness properties (even weak type). Such a combination gives everything in one blow.

We think that a very efficient way to go about things is to first prove a sharp form of a representation theorem working directly from the $T1$ assumptions. This is interesting on its own right,
entails $T1$, gives an explicit equality containing the full dyadic structure of the operator,
and can even be used to transfer sparse bounds, at least in the form sense, from the model dyadic operators to the singular integral. This strategy was employed in the linear setting by Culiuc, Di Plinio and one of us in \cite{CPO}, but
of course they were able to cite the linear representation theorem with $T1$ assumptions from previous literature \cite{Hy2}. It is also to be noted that sparse bounds are remarkably simple to prove for dyadic model operators using the method of \cite{CPO}.

In this paper we, for the first time, prove a representation theorem in the bilinear setting, and we do it starting from the bilinear $T1$ assumptions.
Moreover, we carry out the above strategy in the bilinear setting i.e. we prove sparse domination for our model operators and then transfer them back to the singular integral.
In particular, we get a sparse bilinear $T1$ implying directly the boundedness of singular integrals from $L^p \times L^q$ to $L^r$ for all $1 < p, q < \infty$ and $1/2 < r < \infty$ satisfying $1/p + 1/q = 1/r$, and even the boundedness from $L^1 \times L^1$ to $L^{1/2,\infty}$, just from the $T1$ assumptions. Of course, one can also recover known sharp weighted bounds (see e.g. \cite{LMS}) from sparse domination. It is to be noted though that we prove sparse domination in the trilinear form sense, as such bounds
are easy to transfer using the representation. A caveat regarding weighted bounds is that outside the Banach range the literature currently seems to lack an argument giving \emph{sharp} weighted bounds from form type domination (but such
bounds can be derived using pointwise sparse domination \cite{CR}, \cite{LN}).

The proof of the representation entails finding a dyadic--probabilistic proof technique which produces only simple model operators. Some bilinear dyadic--probabilistic methods were studied by two of us in \cite{MV} in the non-homogeneous setting.
However, there seems to be a plethora of possible ways to decompose things in the bilinear setting, and one
has to be quite careful to really get only nice shifts and nice paraproducts (such that can easily be seen to obey sparse domination). We now move on to formulating some basic definitions and stating our theorems.

A function
$$
K\colon (\R^n \times \R^n \times \R^n) \setminus \Delta \to \C, \qquad \Delta := \{(x,y,z) \in \R^n \times \R^n \times \R^n\colon\, x = y = z\},
$$
is called a standard bilinear Calder\'on--Zygmund kernel if for some $\alpha \in (0,1]$ and $C_K < \infty$ it holds that
$$
|K(x,y,z)| \le \frac{C_K}{(|x-y| + |x-z|)^{2n}},
$$ 
$$
|K(x,y,z)-K(x',y,z)| \le C_K\frac{|x-x'|^{\alpha}}{(|x-y| + |x-z|)^{2n+\alpha}}
$$
whenever $|x-x'| \le \max(|x-y|, |x-z|)/2$,
$$
|K(x,y,z)-K(x,y',z)| \le C_K\frac{|y-y'|^{\alpha}}{(|x-y| + |x-z|)^{2n+\alpha}}
$$
whenever $|y-y'| \le \max(|x-y|, |x-z|)/2$, and
$$
|K(x,y,z)-K(x,y,z')| \le C_K\frac{|z-z'|^{\alpha}}{(|x-y| + |x-z|)^{2n+\alpha}}
$$
whenever $|z-z'| \le \max(|x-y|, |x-z|)/2$. The best constant $C_K$ is denoted by $\|K\|_{\operatorname{CZ}_{\alpha}}$.

Given a standard bilinear Calder\'on--Zygmund kernel $K$
we define
$$
T_{\varepsilon}(f,g)(x)  = \iint_{\max(|x-y|, |x-z|) > \varepsilon} K(x,y,z) f(y)g(z)\,dy\,dz.
$$
The above is well-defined as an absolutely convergent integral if e.g. $f \in L^{p_1}$ and $g \in L^{p_2}$ for some $p_1, p_2 \in [1, \infty)$, since
then
$$
\iint_{\max(|x-y|, |x-z|) > \varepsilon} |K(x,y,z) f(y)g(z)|\,dy\,dz \lesssim \frac{1}{\varepsilon^{m(1/p_1+1/p_2)}} \|f\|_{L^{p_1}} \|g\|_{L^{p_2}}. 
$$
For us a bilinear Calder\'on--Zygmund operator is essentially the family of truncations $(T_{\varepsilon})_{\varepsilon > 0}$. In particular, this means
that boundedness in some $L^p$ spaces is understood in the sense that all $T_{\varepsilon}$ are bounded uniformly in $\varepsilon > 0$.

We shall also define some smoother truncations. Suppose $\varphi \in \mathcal{A}$, where $\mathcal{A}$
consists of smooth functions $\varphi \colon [0, \infty) \to [0,1]$ satisfying that $\varphi = 0$ on $[0,1/2]$, $\varphi = 1$ on $[1,\infty)$ and $\|\varphi'\|_{L^{\infty}} \le 10$.
Define the smoothly truncated singular integrals
$$
T_{\varepsilon}^{\varphi}(f,g)(x) = \iint K_\vare ^\varphi (x,y,z) f(y)g(z) \,dy\,dz, \qquad \varepsilon > 0,
$$
where
$$
K_\vare^\varphi(x,y,z)= K(x,y,z) \varphi\Big( \frac{|x-y|+|x-z|}{\varepsilon} \Big).
$$
The point is that $T^{\varphi}_{\varepsilon}$, $\varepsilon > 0$, are operators with standard bilinear $n$-dimensional kernels (with the kernel
bounds being independent of $\varepsilon$). Moreover, we have
$$
|T_{\varepsilon}(f,g)(x)  - T_{\varepsilon}^{\varphi}(f,g)(x)| \lesssim \mathcal{M}(f,g)(x) := \sup_{r > 0} \bla |f| \bra_{B(x,r)} \bla |g| \bra_{B(x,r)},
$$
where $\bla f \bra_A := \frac{1}{|A|} \int_A f$.
If $0 < \vare_1<\vare_2$ we denote by $T^\varphi_{\vare_1, \vare_2}$ the operator
\begin{equation*}
\begin{split}
T^\varphi_{\vare_1, \vare_2}(f,g)(x)&=T_{\varepsilon_1}^{\varphi}(f,g)(x)-T_{\varepsilon_2}^{\varphi}(f,g)(x) \\
&=\iint K_{\vare_1, \vare_2} ^\varphi (x,y,z) f(y)g(z) \,dy\,dz,
\end{split}
\end{equation*}
where $K_{\vare_1, \vare_2} ^\varphi=K_{\vare_1} ^\varphi-K_{\vare_2} ^\varphi$.

The notation $T^{1*}$ and $T^{2*}$ stand for the adjoints of a bilinear operator $T$, i.e.
$$
\langle T(f, g), h\rangle = \langle T^{1*}(h, g), f\rangle = \langle T^{2*}(f, h),g\rangle.
$$

We can now state our main theorem. For the exact definitions of the various objects and notions (random dyadic grids, bilinear cancellative shifts, bilinear paraproducts, weak boundedness,
$T_{\delta}(1,1)$, sparse collections etc.) see the following two sections.
\begin{thm}\label{thm:main}
Let $K$ be a bilinear Calder\'on--Zygmund kernel so that $\|K\|_{\operatorname{CZ}_{\alpha}}  < \infty$,
and let $(T_{\varepsilon})_{\varepsilon > 0}$ be the corresponding bilinear singular integral.
Assume that
$$
\sup_{\delta >0} [\|T_{\delta}\|_{\operatorname{WBP}} + \|T_{\delta}(1,1)\|_{\operatorname{BMO}} + \|T^{1*}_{\delta}(1,1)\|_{\operatorname{BMO}} + \|T^{2*}_{\delta}(1,1)\|_{\operatorname{BMO}}] < \infty.
$$
Let also $\varphi \in \mathcal{A}$. Then there is a constant $C = C(n,\alpha) < \infty$ so that for all $\vare>0$ and all compactly supported and bounded functions $f, g$ and $h$ it holds that
\begin{align*}
\bla T_{\varepsilon}^{\varphi}(f,g), h \bra =& C(\|K\|_{\operatorname{CZ}_{\alpha}} + \sup_{\delta>0}\|T_{\delta}\|_{\operatorname{WBP}})   
E_{\omega} \sum_{k=0}^\infty \sum_{i=0}^{k}2^{-\alpha k/2} \bla U^{i,k}_{\varepsilon, \varphi, \omega} (f,g),h\bra \\
&+ C(\|K\|_{\operatorname{CZ}_{\alpha}}+ \sup_{\delta>0}\|T_{\delta}(1,1)\|_{\operatorname{BMO}}) E_{\omega} \bla \Pi_{\alpha_0(\varepsilon, \varphi, \omega)}(f,g), h \bra \\
&+ C(\|K\|_{\operatorname{CZ}_{\alpha}}+ \sup_{\delta>0}\|T^{1*}_{\delta}(1,1)\|_{\operatorname{BMO}})  E_{\omega} \bla \Pi^{1*}_{\alpha_1(\varepsilon, \varphi, \omega)}(f,g), h \bra \\
&+ C(\|K\|_{\operatorname{CZ}_{\alpha}}+ \sup_{\delta>0}\|T^{2*}_{\delta}(1,1)\|_{\operatorname{BMO}}) E_{\omega} \bla \Pi^{2*}_{\alpha_2(\varepsilon, \varphi, \omega)}(f,g), h \bra,
\end{align*}
where each $U^{i,k}_{\varepsilon, \varphi, \omega}$
is a sum of  cancellative bilinear shifts $S^{i,i,k}_{\varepsilon, \varphi, \omega}$, $S^{i,i+1,k}_{\varepsilon, \varphi, \omega}$  and adjoints of such operators,
and $\Pi_{\alpha}$ stands for a bilinear paraproduct with $\alpha$ as in \eqref{eq:BMO2}.
For a fixed $\omega$ the operators above are defined using the dyadic lattice $\calD_\omega$.
\end{thm}
The following corollary follows from the sparse domination of shifts and paraproducts (see Section \ref{sec:SFPARSE}), and the trivial sparse bound for $\mathcal{M}$.
\begin{cor}\label{cor:main}
There exist dyadic grids $\mathcal{D}_i$, $i = 1, \ldots, 3^n$, with the following property. Let $\eta \in (0,1)$.
For compactly supported and bounded functions $f, g$ and $h$
there is a dyadic grid $\calD_i$ and an $\eta$-sparse collection $\mathcal{S}=\mathcal{S}(f,g,h, \eta) \subset \calD_i$  so that the following holds.

Let $K$ be any standard bilinear Calder\'on--Zygmund kernel and $(T_{\varepsilon})_{\varepsilon > 0}$ be the corresponding bilinear singular integral. Then we have
$$
\sup_{\varepsilon > 0} |\bla T_{\varepsilon} (f,g), h \bra|
\le C_{T,K} \Lambda_{\mathcal{S}}(f, g, h),
$$
where
\begin{align*}
C_{T,K} := C(\|K\|_{\operatorname{CZ}_{\alpha}}+ \sup_{\varepsilon>0}\|T_{\varepsilon}(1,1)\|_{\operatorname{BMO}} &+  \sup_{\varepsilon>0}\|T^{1*}_{\varepsilon}(1,1)\|_{\operatorname{BMO}} \\
&+  \sup_{\varepsilon>0}\|T^{2*}_{\varepsilon}(1,1)\|_{\operatorname{BMO}} + \sup_{\varepsilon>0}\|T_{\varepsilon}\|_{\operatorname{WBP}})
\end{align*}
for some $C = C(n,\alpha) < \infty$ and
$$
\Lambda_{\mathcal{S}}(f, g, h) := \sum_{Q\in\mathcal{S}}|Q| \bla |f|\bra_Q \bla |g|\bra_Q \bla |h|\bra_Q.
$$
\end{cor}

\subsection*{Additional notation}
We write $A \lesssim B$, if there is an absolute constant $C>0$ (depending only on some fixed constants like $n$ and $\alpha$ etc.) so that $A \leq C B$.
Moreover, $A \lesssim_{\tau} B$ means that the constant $C$ can also depend on some relevant given parameter $\tau > 0$.
We may also write $A \sim B$ if $B \lesssim A \lesssim B$.

We then define some notation related to cubes. If $Q$ and $R$ are two cubes we set:
\begin{itemize} 
\item $\ell(Q)$ is the side-length of $Q$;
\item If $a>0$, we denote by $aQ$ the cube that is concentric with $Q$ and has sidelength $a\ell(Q)$; 
\item $d(Q,R) = \dist(Q,R)$ denotes the distance between the cubes $Q$ and $R$;
\item $\text{ch}(Q)$ denotes the dyadic children of $Q$;
\item If $Q$ is in a dyadic grid, then $Q^{(k)}$ denotes the unique dyadic cube $S$ in the same grid so that $Q \subset S$ and $\ell(S) = 2^k\ell(Q)$;
\item If $\calD$ is a dyadic grid, then $\calD_k = \{Q \in \calD\colon\, \ell(Q) = 2^{-k}\}$;
\end{itemize}
The notation $\langle f, g\rangle$ stands for the pairing $\int fg$.

The following maximal functions are also used:
\begin{align*}
M_{\calD} f(x) &= \sup_{Q \in \calD} 1_Q(x) \langle |f| \rangle_Q \qquad (\calD \textup{ is a dyadic grid}); \\
M f(x) &= \sup_{r>0} \, \langle |f| \rangle_{B(x,r)}. 
\end{align*}
Here $B(x,r) = \{y\colon\, |x-y| < r\}$.
The bilinear variants are defined in the natural way, e.g.
$$
\mathcal{M}(f,g)(x) = \sup_{r>0} \, \langle |f| \rangle_{B(x,r)} \langle |g| \rangle_{B(x,r)}.
$$

\subsection*{Acknowledgements}
K.L. is supported by Juan de la Cierva - Formaci\'on 2015 FJCI-2015-24547, the Basque Government through the BERC 2014-2017 program and
by Spanish Ministry of Economy and Competitiveness MINECO: BCAM Severo Ochoa excellence accreditation SEV-2013-0323.
H.M. is supported by the Academy of Finland through the grants 294840 and 306901, and is a member of the Finnish Centre of Excellence in Analysis and Dynamics Research.
Y.O. was in residence at the Mathematical Sciences Research Institute in Berkeley, California, during the Spring 2017 semester when this work was carried out, and was
supported by the National Science Foundation under Grant No. DMS-1440140.
E.V. is a member of the Finnish Centre of Excellence in Analysis and Dynamics Research.

\section{Basic definitions}
\subsection{Random dyadic grids, martingales, Haar functions}
Let $\omega = (\omega^i)_{i \in \Z}$, where $\omega^i \in \{0,1\}^n$. Let $\mathcal{D}_0$ be the standard dyadic grid on $\R^n$.
We define the new dyadic grid 
$$
\mathcal{D}_{\omega} = \Big\{I + \sum_{i:\, 2^{-i} < \ell(I)} 2^{-i}\omega^i: \, I \in \mathcal{D}_0\Big\} = \{I + \omega: \, I \in \mathcal{D}_0\},
$$
where we simply have defined $I + \omega := I + \sum_{i:\, 2^{-i} < \ell(I)} 2^{-i}\omega^i$. There is a natural product probability measure $\mathbb{P}_{\omega} = \mathbb{P}$ on $(\{0,1\}^n)^{\Z}$ -- this gives us the notion
of random dyadic grids $\omega \mapsto \mathcal{D}_{\omega}$.

A cube $I \in \mathcal{D} = \mathcal{D}_{\omega}$ is called bad if there exists such a cube $J \in \mathcal{D}$ that $\ell(J) \ge 2^r \ell(I)$ and 
$$
d(I, \partial J) \le \ell(I)^{\gamma}\ell(J)^{1-\gamma}.
$$
Here $\gamma = \alpha/(2[2n + \alpha])$, where $\alpha > 0$ appears in the kernel estimates. Otherwise a cube is called good.
We note that
$\pi_{\textrm{good}} := \mathbb{P}_{\omega}(I + \omega \textrm{ is good})$ is independent of the choice of $I \in \mathcal{D}_0$. The appearing parameter $r$ is a large enough fixed constant so that $\pi_{\textrm{good}} > 0$.
Moreover, for a fixed $I \in \mathcal{D}_0$
the set $I + \omega$ depends on $\omega^i$ with $2^{-i} < \ell(I)$, while the goodness of $I + \omega$ depends on $\omega^i$ with $2^{-i} \ge \ell(I)$. These notions are thus independent by the product probability structure.

For $I \in \calD$ and a locally integrable function $f$ we define the martingale difference
$$
\Delta_I f = \sum_{I' \in \textup{ch}(I)} \big[ \bla f \bra_{I'} -  \bla f \bra_{I} \big] 1_{I'}.
$$
We have the standard estimate
$$
\Big\| \Big( \sum_{I \in \calD}  |\Delta_I f|^2  \Big)^{1/2} \Big\|_{L^p} \sim \|f\|_{L^p}, \qquad 1 < p < \infty.
$$
Writing $I = I_1 \times \cdots \times I_n$ we can define the Haar function $h_I^{\eta}$, $\eta = (\eta_1, \ldots, \eta_n) \in \{0,1\}^n$, by setting
\begin{displaymath}
h_I^{\eta} = h_{I_1}^{\eta_1} \otimes \cdots \otimes h_{I_n}^{\eta_n}, 
\end{displaymath}
where $h_{I_i}^0 = |I_i|^{-1/2}1_{I_i}$ and $h_{I_i}^1 = |I_i|^{-1/2}(1_{I_{i, l}} - 1_{I_{i, r}})$ for every $i = 1, \ldots, n$. Here $I_{i,l}$ and $I_{i,r}$ are the left and right
halves of the interval $I_i$ respectively. If $\eta \ne 0$ the Haar function is cancellative: $\int h_I^{\eta} = 0$. We have that
$$
\Delta_I f = \sum_{\eta \in \{0,1\}^n \setminus \{0\}} \bla f, h_I^{\eta} \bra h_I^{\eta},
$$
but for convenience we understand that the $\eta$ summation is suppressed and simply write
$$
\Delta_I f = \bla f, h_I \bra h_I.
$$
In this paper $h_I$ always denotes a cancellative Haar function (i.e. $h_I = h_I^{\eta}$ for some $\eta \ne 0$). A non-cancellative Haar function is explicitly denoted by
$h_I^0$.

\subsection{Testing conditions: BMO and WBP}
Let $K$ be a standard bilinear Calder\'on-Zygmund kernel, and let $\{T_\varepsilon\}_{\varepsilon>0}$ be the related family of truncated operators. We recall a usual interpretation of  $T_{\varepsilon}(1,1)$ and what is means that it belongs $\BMO$.

Fix some $\varepsilon>0$. Let $R\subset \R^n$ be a closed cube and let $\phi$ be an $L^\infty$ function supported in $R$ such that $\int \phi=0$. 
Let $C=C(\vare)\ge 3$ be any large constant so that $2^{-1}(C-1)\ell(R)>\vare$, whence $|x-y| > \vare$ for all $x \in R$ and $y \not \in CR$. We define
\begin{equation}\label{eq:DefPair}
\begin{split}
\bla T_\varepsilon (1,1), \phi \bra
&:= \bla T_\varepsilon (1_{CR}, 1_{CR}), \phi \bra \\
&+\iiint  \big(K(x,y,z)-K(c_R,y,z)\big)1_{(CR \times CR)^c}(y,z)\phi(x)\, dy\, dz \,dx.
\end{split}
\end{equation}
Applying the $x$-H\"older estimate of the kernel it is seen that the integral is absolutely convergent.
It is straightforward to check that the right hand side of \eqref{eq:DefPair} is independent of the cube $R$ and the constant $C$ as long as
$\phi$ is supported in $R$ and $2^{-1}(C-1)\ell(R)>\vare$, $C \ge 3$.

If $\varphi \in \mathcal{A}$ and   $\phi$ is as above, we define 
\begin{equation}\label{eq:DefPairSmooth}
\begin{split}
\bla T_\varepsilon^\varphi(1,1), \phi \bra
&:= \bla T_\varepsilon ^\varphi (1_{CR}, 1_{CR}), \phi \bra \\
&+\iiint  \big(K^\varphi_\vare(x,y,z)-K^\varphi_\vare(c_R,y,z)\big)1_{(CR \times CR)^c}(y,z)\phi(x)\, dy\, dz \,dx
\end{split}
\end{equation}
for any closed cube $R$ containing the support of $\phi$ and any $C\ge3$, say.

\begin{defn}
Let $\vare>0$. Suppose $K$ is a standard bilinear Calder\'on-Zygmund kernel, and let $T_\varepsilon$ be the related truncated operator. We say that $T_\vare(1,1)$ is in $\BMO$, and write $T_\vare (1,1) \in \BMO$, if there exists a constant $C$ so  that
for all closed cubes $R$ and all functions $\phi$ supported in $R$ such that $\| \phi \|_{L^\infty} \le 1$ and $\int \phi =0$ there holds
\begin{equation}\label{eq:DefBMO}
 \frac{\big| \bla T_\vare(1,1), \phi \bra \big|}{|R|}    \le C.
\end{equation}
We denote the smallest constant $C$ in \eqref{eq:DefBMO} by $\| T_\vare(1,1)\|_{\BMO}$.

If $\varphi \in \mathcal{A}$, the corresponding definition for the smoothly truncated operator $T^\varphi_\vare$ is obtained just by replacing $T_\vare$ by $T^\varphi_\vare$. 
\end{defn}

In the representation theorem we will assume that $T_\vare(1,1) \in \BMO$. The following simple lemma shows that the conditions $T_\vare(1,1) \in \BMO$ and $T_\vare^\varphi(1,1) \in \BMO$ are equivalent.

\begin{lem}\label{lem:BMOTransfer}
Suppose $K$ is a standard bilinear Calder\'on-Zygmund kernel and let $\vare>0$ and $\varphi \in \mathcal{A}$. Then 
$$
\| T_\vare^\varphi (1,1)\|_{\BMO} \le C \big(\| K \|_{\operatorname{CZ}_\alpha}+ \| T_\vare(1,1)\|_{\BMO}\big) \quad 
$$
and
$$
\| T_\vare (1,1)\|_{\operatorname{BMO}} \le C \big(\| K \|_{\operatorname{CZ}_\alpha}+ \| T^\varphi_\vare(1,1)\|_{\operatorname{BMO}}\big).
$$

\end{lem}

\begin{proof}
Fix a closed cube $R$ and a function $\phi$ supported in $R$ such that $\| \phi \|_{L^\infty} \le 1$ and $\int \phi=0$.
Then, using the definitions \eqref{eq:DefPair} and \eqref{eq:DefPairSmooth}, one sees that
\begin{equation*}
\begin{split}
\big| \bla T_\vare(1,1), \phi \bra- \bla T^\varphi_\vare (1,1), \phi \bra \big|
&=\big| \bla T_\vare(1_{CR},1_{CR}), \phi \bra- \bla T^\varphi_\vare (1_{CR},1_{CR}), \phi \bra \big|\\
&\lesssim \bla \calM(1_{CR},1_{CR}), |\phi| \bra 
\le \int |\phi| dx \le |R|.
\end{split}
\end{equation*}

The claim follows from this estimate.
\end{proof}

For the convenience of the reader we state the following lemma on the equivalence of some BMO type conditions -- 
although $T_\vare^\varphi(1,1)$ is not stricly speaking a function, the lemma nevertheless follows from John--Nirenberg by standard arguments. Therefore, the
paraproducts we will encounter can be made to obey the normalisation in \eqref{eq:BMO2}.

\begin{lem}\label{lem:JN}
Suppose $K$ is a standard bilinear Calder\'on-Zygmund kernel and let $\vare>0$ and $\varphi \in \mathcal{A}$. 
Suppose $\calD$ is a dyadic lattice.  Then
$$
\sup_{R \in \calD}\frac{1}{|R|} \sum_{\substack{Q \in \calD \\ Q \subset R}} \big|\bla T_\vare^\varphi(1,1), h_Q \bra \big|^2
\le C \| T_\vare^\varphi(1,1)\|_{\BMO}^2
$$
for some absolute constant $C$.
\end{lem}

Next, we give the definition of weak boundedness property.
\begin{defn}
The weak boundedness property constant $\|T_{\varepsilon}\|_{\operatorname{WBP}}$ is the best constant $C$ so that the inequality
$$
|\bla T_{\varepsilon}(1_I, 1_I), 1_I\bra| \le C|I|
$$
holds for all cubes $I \subset \R^n$.
\end{defn}

\subsection{Sparse collections}
A collection $\mathcal{S}$ of cubes is said to be \emph{$\eta$-sparse} (or just sparse), $0<\eta<1$, if for any $Q\in\mathcal{S}$ there exists $E_Q\subset Q$ so that $|E_Q|>\eta|Q|$ and $\{E_Q:\,Q\in\mathcal{S}\}$ are pairwise disjoint.
The definition does not require the cubes to be part of some fixed dyadic grid. Although, it can be convenient to know that in Corollary \ref{cor:main} the sparse family $\mathcal{S}$ can always
be found inside one of the fixed dyadic grids $\calD_i$, $\#i \lesssim 1$.

\section{Bilinear shifts}
In this section all cubes are part of some fixed dyadic grid $\calD$. We will introduce certain cancellative shifts and paraproducts in this section.
We will also show their boundedness $L^p \times L^q \to L^r$ in the simple case $1 < p, q, r < \infty$ satisfying $1/p + 1/q = 1/r$.
The restriction $r > 1$ can be lifted after we have shown the sparse domination (see Section \ref{sec:SFPARSE}).
\subsection{Cancellative bilinear shifts}
Define for $i,j,k \ge 0$ the bilinear shift $(f,g) \mapsto S^{i,j,k}(f,g)$ by setting
$$
S^{i,j,k}(f,g) = \sum_Q A^{i,j,k}_Q(f,g),
$$
where
$$
A^{i,j,k}_Q(f,g) = \mathop{\mathop{\mathop{\sum_{I, J, K \subset Q}}_{\ell(I) = 2^{-i}\ell(Q)}}_{\ell(J) = 2^{-j} \ell(Q)}}_{\ell(K) = 2^{-k}\ell(Q)} \alpha_{I,J,K,Q} \bla f, \tilde h_I \bra \bla g, \tilde h_J \bra h_K
$$
and
$$
(\tilde h_I, \tilde h_J) \in \big\{ (h_I, h_J),(h_I^0, h_J)(h_I, h_J^0)\big\}.
$$
We also demand that
$$
| \alpha_{I,J,K,Q} | \le \frac{|I|^{1/2} |J|^{1/2} |K|^{1/2} } { |Q|^2}.
$$ 
Such a shift will be considered to be a \emph{cancellative bilinear shift}. 
Also the duals of these operators will be used in the representation.

Let $1 < p, q, r < \infty$ be such that $1/p + 1/q = 1/r$.
We show that
$$
\| S^{i,j,k}(f,g) \|_{L^r}  \lesssim \|f\|_{L^p} \|g\|_{L^q}
$$
with the constant independent of the shift in question, and only depending on $p,q,r$.
To do this, we may assume without loss of generality that for example $\tilde h_I=h_I$ for all $I$ (a general shift can be split into two shifts
where  $\tilde h_I=h_I$ for all $I$ in one of them and $\tilde h_J=h_J$ for all $J$ in the other).
Notice that we have the pointwise estimate
$$
|A^{i,j,k}_Q(f,g)| \le \bla |f| \bra_Q  \bla |g| \bra_Q 1_Q.
$$
Define also
$$
D_Q^i f = \mathop{\sum_{I \subset Q}}_{\ell(I) = 2^{-i}\ell(Q)} \bla f , h_I \bra h_I.
$$
Since $A^{i,j,k}_Q(f,g) = A^{i,j,k}_Q(D_Q^i f,g)$, we have 
$$
|A^{i,j,k}_Q(f,g)| \le M_{\calD} g \bla |D_Q^i f| \bra_Q  1_Q.
$$

Notice that
$$
\Big\| \Big( \sum_Q  |D_Q^i f|^2  \Big)^{1/2} \Big\|_{L^p} = \Big\| \Big( \sum_I  |\Delta_I f|^2  \Big)^{1/2} \Big\|_{L^p} \sim \|f\|_{L^p}, \qquad 1 < p < \infty.
$$
Let $1 < p, q, r < \infty$ be such that $1/p + 1/q = 1/r$. Using the above we see that
$$
\| S^{i,j,k}(f,g) \|_{L^r} \sim \Big\| \Big( \sum_Q  |D_Q^k(S^{i,j,k}(f,g)) |^2  \Big)^{1/2} \Big\|_{L^r} = \Big\| \Big( \sum_Q  |A^{i,j,k}_Q(f,g) |^2  \Big)^{1/2} \Big\|_{L^r}.
$$
Now, we have
\begin{align*}
\Big\| \Big( \sum_Q  |A^{i,j,k}_Q(f,g)) |^2  \Big)^{1/2} \Big\|_{L^r} &\le \Big\| M_{\calD} g \Big( \sum_Q  \bla |D_Q^i f| \bra_Q^2  1_Q  \Big)^{1/2} \Big\|_{L^r} \\
&\le \Big\|  \Big( \sum_Q  \bla |D_Q^i f| \bra_Q^2  1_Q  \Big)^{1/2} \Big\|_{L^p} \|M_{\calD} g\|_{L^q} \\
&\lesssim \Big\| \Big( \sum_Q  |D_Q^i f|^2  \Big)^{1/2} \Big\|_{L^p} \|g\|_{L^q} \lesssim \|f\|_{L^p} \|g\|_{L^q}.
\end{align*}

\subsection{Bilinear paraproduct}
Let $\alpha=\{\alpha_K\}_{K \in \calD}$ be sequence of complex numbers such that
\begin{equation}\label{eq:BMO2}
\frac{1}{|K_0|}\sum_{K \colon K \subset K_0} |\alpha_K|^2 \le 1
\end{equation}
for all $K_0 \in \calD$.
We define the bilinear paraproduct
$$
\Pi_{\alpha}(f,g) = \sum_{K} \alpha_K \bla f \bra_K \bla g \bra_K h_K.
$$
To deal with this it is useful to recall
the usual (linear) paraproduct
$$
\pi_{\alpha} f = \sum_{K} \alpha_K \bla f \bra_K h_K.
$$
It is well known that $\pi_{\alpha} \colon L^r \to L^r$ boundedly for $1 < r < \infty$ because of the condition \eqref{eq:BMO2}. An elegant way to do this directly in $L^r$ is in \cite{HH}.
It follows that
$\Pi_{\alpha} \colon L^p \times L^q \to L^r$ boundedly for $1 < p, q, r < \infty$ satisfying $1/r = 1/p + 1/q$. Indeed, it holds that
\begin{align*}
\| \Pi_{\alpha}(f,g) \|_{L^r}  &\sim \Big\| \Big( \sum_{K} |\alpha_K|^2 |\bla f \bra_K|^2 |\bla g \bra_K|^2 \frac{1_K}{|K|} \Big)^{1/2} \Big\|_{L^r} \\
&\le \Big\| \Big( \sum_{K} |\alpha_K|^2\bla \calM_{\calD}(f,g) \bra_K^2 \frac{1_K}{|K|} \Big)^{1/2} \Big\|_{L^r} \\
&\sim \| \pi_{\alpha}( \calM_{\calD}(f,g) ) \|_{L^r}
\lesssim \| \calM_{\calD}(f,g) \|_{L^r} \lesssim \|f\|_{L^p} \|g\|_{L^q}.
\end{align*}

\section{Proof of the bilinear representation theorem, Theorem \ref{thm:main}}
Consider an arbitrary $\varepsilon_1 > 0$ and let $f$, $g$ and $h$ be bounded functions with compact support.
For the moment, let $\varepsilon_2 > \varepsilon_1$ be arbitrary, and write
$T=T^\varphi_{\vare_1, \vare_2}$ and $K=K^\varphi_{\vare_1, \vare_2}$. This is an a priori bounded operator (for example in the $L^4 \times L^4 \to L^2$ sense), which makes
the calculations below legit. We will decompose $\bla T(f , g), h \bra$ first, and take the limit $\epsilon_2 \to \infty$ at the end.

Begin by decomposing $\bla T(f , g), h \bra$ as
\begin{equation*}
\begin{split}
\bla T(f , g), h \bra
&= E_{\omega} \sum_{K \in \calD_{\omega}} \sum_{\substack{I \in \calD_{\omega} \\ \ell(K) \leq \ell(I)  }}
\sum_{\substack{J \in \calD_{\omega} \\ \ell(K) \leq \ell(J) }} \bla T(\Delta_I f , \Delta_Jg), \Delta_K h \bra \\
& + E_{\omega} \sum_{I \in \calD_{\omega}} \sum_{\substack{J \in \calD_{\omega} \\ \ell(I) \leq \ell(J)  }}
\sum_{\substack{K \in \calD_{\omega} \\ \ell(I) < \ell(K) }} \bla T^{1*}(\Delta_K h , \Delta_Jg), \Delta_I f \bra \\
& +E_{\omega} \sum_{J \in \calD_{\omega}} \sum_{\substack{I \in \calD_{\omega} \\ \ell(J) < \ell(I)  }}
\sum_{\substack{K \in \calD_{\omega} \\ \ell(J) < \ell(K) }} \bla T^{2*}(\Delta_I f , \Delta_Kh), \Delta_J g \bra =: \Sigma^1 + \Sigma^2 + \Sigma^3.
\end{split}
\end{equation*}
We focus on the first sum $\Sigma^1$, and at this point write
\begin{align*}
\sum_{K \in \calD_{\omega}} \sum_{\substack{I \in \calD_{\omega} \\ \ell(K) \leq \ell(I)  }}
\sum_{\substack{J \in \calD_{\omega} \\ \ell(K) \leq \ell(J) }} \bla T(\Delta_I f , \Delta_Jg), \Delta_K h \bra
= \sum_{K \in \calD_{\omega}}  \bla T(E^{\omega}_{\ell(K)/2} f , E^{\omega}_{\ell(K)/2} g), \Delta_K h \bra,
\end{align*}
where
$$
E^{\omega}_{\ell(K)/2} f = \mathop{\sum_{I \in \calD_{\omega}}}_{\ell(I) = \ell(K) / 2} 1_I \bla f \bra_I.
$$
The point of doing this is to gain the needed independence for the argument below (this seems to be a new simpler way to add goodness than in \cite{Hy2}, and is straightforward to use also in this bilinear setting).
Write now $\calD_{\omega} = \calD_0 + \omega$ to the end that
$$
\sum_{K \in \calD_{\omega}}  \bla T(E^{\omega}_{\ell(K)/2} f , E^{\omega}_{\ell(K)/2} g), \Delta_K h \bra = \sum_{K \in \calD_0}  \bla T(E^{\omega}_{\ell(K)/2} f , E^{\omega}_{\ell(K)/2} g), \Delta_{K+\omega} h \bra.
$$
Next, we write
\begin{align*}
\Sigma^1 = E_{\omega} & \sum_{K \in \calD_0}  \bla T(E^{\omega}_{\ell(K)/2} f , E^{\omega}_{\ell(K)/2} g), \Delta_{K+\omega} h \bra \\
&= \frac{1}{\pi_{\textup{good}}} \sum_{K \in \calD_0} E_{\omega}[1_{\textup{good}}(K+\omega)] E_{\omega} [\bla T(E^{\omega}_{\ell(K)/2} f , E^{\omega}_{\ell(K)/2} g), \Delta_{K+\omega} h \bra] \\ 
&= \frac{1}{\pi_{\textup{good}}} E_{\omega} \sum_{K \in \calD_{\omega,\, \textup{good}}}  \bla T(E^{\omega}_{\ell(K)/2} f , E^{\omega}_{\ell(K)/2} g), \Delta_K h \bra =: \frac{1}{\pi_{\textup{good}}} E_{\omega} \Sigma^1(\omega),
\end{align*}
where we used independence: $1_{\textup{good}}(K+\omega)$ depends on $\omega_j$ for $2^{-j} \ge \ell(K)$ while
$E^{\omega}_{\ell(K)/2} f$ depends on $\omega_j$ for $2^{-j} < \ell(K)/2 < \ell(K)$, same for $E^{\omega}_{\ell(K)/2} g$, and
$\Delta_K h$ depends on $\omega_j$ for $2^{-j} < \ell(K)$.

Fix $\omega$ and let $\calD_{\omega} = \calD$. We will now start finding the shift structure in the sum $\Sigma^1(\omega)$ i.e.
$$
\sum_{K \in \calD_{\textup{good}}}  \sum_{\substack{I \in \calD \\ \ell(K) \leq \ell(I)  }}
\sum_{\substack{J \in \calD \\ \ell(K) \leq \ell(J) }} \bla T(\Delta_I f , \Delta_Jg), \Delta_K h \bra.
$$
The double sum $\sum_{\substack{I \in \calD \\ \ell(K) \leq \ell(I)  }} \sum_{\substack{J \in \calD \\ \ell(K) \leq \ell(J) }}$ 
can be organised as
$$
\sum_{\substack{I \in \calD \\ \ell(K) \leq \ell(I)  }} \sum_{\substack{J \in \calD \\ \ell(I) \leq \ell(J)  }}
+\sum_{\substack{J \in \calD \\ \ell(K) \leq \ell(J) }} \sum_{\substack{I \in \calD \\ \ell(J) < \ell(I)  }}.
$$
This leads to the fact that
\begin{align*}
\sum_{K \in \calD_{\textup{good}}}  &\sum_{\substack{I \in \calD \\ \ell(K) \leq \ell(I)  }}
\sum_{\substack{J \in \calD \\ \ell(K) \leq \ell(J) }} \bla T(\Delta_I f , \Delta_Jg), \Delta_K h \bra  \\
&=  \sum_{K \in \calD_{\textup{good}}} \sum_{\substack{I \in \calD \\ \ell(K) \leq \ell(I)  }} \bla T(\Delta_I f , E_{\ell(I)/2} g), \Delta_K h \bra \\
&+ \sum_{K \in \calD_{\textup{good}}} \sum_{\substack{J \in \calD \\ \ell(K) \leq \ell(J) }}  \bla T(E_{\ell(J)} f , \Delta_Jg), \Delta_K h \bra =: \sigma^1 + \sigma^2.
\end{align*}
We will now mostly focus on the part
\begin{align*}
\sigma^1 = \sum_{K \in \calD_{\textup{good}}} &\sum_{\substack{I \in \calD \\ \ell(K) \leq \ell(I)  }} \bla T(\Delta_I f , E_{\ell(I)/2} g), \Delta_K h \bra \\
&= \sum_{K \in \calD_{\textup{good}}} \sum_{\substack{I \in \calD \\ \ell(K) \leq \ell(I)  }} \mathop{\sum_{J \in \calD}}_{\ell(J) = \ell(I) / 2} 
 \bla T(\Delta_I f ,1_J \bla g \bra_J ), \Delta_K h \bra.
\end{align*}
However, to get a simple paraproduct it is crucial to combine i.e. sum up the paraproduct parts from these two parts $\sigma^1$ and $\sigma^2$.

\subsection*{Step I: separated part}
In this section we consider
\begin{align*}
\sigma^1_1&:= \sum_{K \in \calD_{\textup{good}}} \sum_{\substack{I,J \in \calD \colon \\ \ell(K)\le \ell (I)=2\ell(J) \\ \max( d(K,I), d(K,J)) > \ell(K)^\gamma \ell(J)^{1-\gamma}}} 
\bla T(\Delta_I f ,1_J \bla g \bra_J ), \Delta_K h \bra \\
&= \sum_{K \in \calD_{\textup{good}}} \sum_{\substack{I,J \in \calD \colon \\ \ell(K)\le \ell (I)=2\ell(J) \\ \max( d(K,I), d(K,J)) > \ell(K)^\gamma \ell(J)^{1-\gamma}}} 
\bla T(h_I, h_J^0), h_K \bra \bla f, h_I \bra \bla g, h_J^0\bra \bla h, h_K \bra.
\end{align*}
We need the existence of certain nice parents, the proof in the bilinear setting is essentially the same as in \cite{Hy2}.
\begin{lem}
For $I, J, K$ as in $\sigma^1_1$ there exists a cube $Q \in \calD$ so that $I \cup J \cup K \subset Q$ and
$$
\max(d(K,I), d(K,J)) \gtrsim \ell(K)^{\gamma}\ell(Q)^{1-\gamma}.
$$
\end{lem}
\begin{proof}
Let $Q \in \calD$ be the minimal parent of $K$ for which both of the following two conditions hold:
\begin{itemize}
\item $\ell(Q) \ge 2^r \ell(K)$;
\item $\max(d(K,I), d(K,J)) \le \ell(K)^{\gamma}\ell(Q)^{1-\gamma}$.
\end{itemize}
Since $\ell(Q) \ge 2^r \ell(K)$, the goodness of $K$ gives that
$$
\ell(K)^{\gamma}\ell(Q)^{1-\gamma} < d(K, Q^c).
$$
If we would have that $I \subset Q^c$ or $J \subset Q^c$ we would get
$$
\ell(K)^{\gamma}\ell(Q)^{1-\gamma} < \max(d(K,I), d(K,J)) \le \ell(K)^{\gamma}\ell(Q)^{1-\gamma},
$$
which is a contradiction. Therefore, we have $I \cap Q \ne \emptyset$ and $J \cap Q \ne \emptyset$. Moreover, we have
$$
\ell(K)^{\gamma}\ell(J)^{1-\gamma} <  \max( d(K,I), d(K,J)) \le \ell(K)^{\gamma}\ell(Q)^{1-\gamma}
$$
implying that $\ell(Q) > \ell(J)$, and so also $\ell(Q) \ge \ell(I)$. This implies $I \cup J \cup K \subset Q$.

It remains to note that the estimate $\max(d(K,I), d(K,J)) \gtrsim \ell(K)^{\gamma}\ell(Q)^{1-\gamma}$ is a trivial consequence of the minimality of $Q$.
Indeed, there is something to check only if $Q$ is minimal because $\ell(Q) \lesssim \ell(K)$. But then $\ell(Q) \lesssim \ell(J)$ and we get
$$
\ell(K)^{\gamma}\ell(Q)^{1-\gamma} \lesssim \ell(K)^{\gamma}\ell(J)^{1-\gamma} < \max( d(K,I), d(K,J)).
$$
\end{proof}
For $I, J, K$ as in $\sigma^1_1$ we let $Q = I \vee J \vee K$ be the minimal cube $Q \in \calD$ so that $I \cup J \cup K \subset Q$. We then know that
\begin{equation}\label{eq:parentSEP}
\max(d(K,I), d(K,J)) \gtrsim \ell(K)^{\gamma}\ell(Q)^{1-\gamma}.
\end{equation}
Let us write
$$
\sigma^1_1 = \sum_{k=0}^{\infty} \sum_{i = 0}^{k} \sum_{Q \in \calD} \mathop{\mathop{\mathop{\sum_{I, J \in \calD, \, K \in \calD_{\good}}}_{\max( d(K,I), d(K,J)) > \ell(K)^{\gamma}\ell(J)^{1-\gamma}}}_{2\ell(J) = \ell(I) = 2^{-i}\ell(Q), \, \ell(K) = 2^{-k}\ell(Q)}}_{I \vee J \vee K = Q}
\bla T(h_I, h_J^0), h_K \bra \bla f, h_I \bra \bla g, h_J^0\bra \bla h, h_K \bra.
$$
Next, we define
$$
\alpha_{I,J,K,Q} = \frac{ \bla T(h_I, h_J^0), h_K \bra }{C (\ell(K) / \ell(Q))^{\alpha/2}}
$$
if $I, J \in \calD, \, K \in \calD_{\good}$, $\max( d(K,I), d(K,J)) > \ell(K)^{\gamma}\ell(J)^{1-\gamma}$, $\ell(K) \le \ell(I) = 2\ell(J)$ and $I \vee J \vee K = Q$,
and $\alpha_{I,J,K,Q} = 0$ otherwise. We can then write for fixed $k \ge 0$ and $i \le k$ that
\begin{align*}
\sum_{Q \in \calD} & \mathop{\mathop{\mathop{\sum_{I, J \in \calD, \, K \in \calD_{\good}}}_{\max( d(K,I), d(K,J)) > \ell(K)^{\gamma}\ell(J)^{1-\gamma}}}_{2\ell(J) = \ell(I) = 2^{-i}\ell(Q), \, \ell(K) = 2^{-k}\ell(Q)}}_{I \vee J \vee K = Q}
\bla T(h_I, h_J^0), h_K \bra \bla f, h_I \bra \bla g, h_J^0\bra h_K \\
&= C2^{-\alpha k / 2} \sum_{Q \in \calD} \mathop{\mathop{\mathop{\sum_{I, J, K \subset Q}}_{\ell(I) = 2^{-i}\ell(Q)}}_{\ell(J) = 2^{-i-1} \ell(Q)}}_{\ell(K) = 2^{-k}\ell(Q)} \alpha_{I,J,K,Q} \bla f, h_I \bra \bla g, h_J^0 \bra h_K =: C2^{-\alpha k / 2} S^{i,i+1,k}(f,g),
\end{align*}
which gives
$$
\sigma^1_1 =  C \sum_{k=0}^{\infty} \sum_{ i=0}^{k} 2^{-k\alpha/2} \bla S^{i,i+1,k}(f,g), h \bra.
$$

It remains to verify that
$$
|\alpha_{I,J,K,Q}| \le \frac{|I|^{1/2} |J|^{1/2} |K|^{1/2} } { |Q|^2}
$$
for an appropriate choice of the constant $C$ depending on the kernel estimates. We fix $I, J, K, Q$ so that $\alpha_{I,J,K,Q} \ne 0$.
Notice that
$|x-c_K| \le \ell(K)/2$ (we are using the $\ell^{\infty}$ distance) for $x \in K$ while
$$
\max(|x-y|, |x-z|) \ge \max( d(K,I), d(K,J)) > \ell(K)^{\gamma}\ell(J)^{1-\gamma} \ge 2^{\gamma} \frac{\ell(K)}{2} \ge 2^{\gamma}|x-c_K|
$$
for $x \in K$, $y \in I$ and $z \in J$. Therefore, we have by the H\"older estimate in the $x$ variable and the estimate \eqref{eq:parentSEP} that
\begin{align*}
|\bla T(h_I, h_J^0), h_K \bra| &\lesssim \|h_I\|_{L^1} \|h_J^0\|_{L^1} \|h_K\|_{L^1}  \frac{\ell(K)^{\alpha}}{\max( d(K,I), d(K,J))^{2n+\alpha}} \\
&\lesssim |I|^{1/2} |J|^{1/2} |K|^{1/2} \frac{\ell(K)^{\alpha}}{(\ell(K)^{\gamma}\ell(Q)^{1-\gamma})^{2n+\alpha}} \\
&= \frac{|I|^{1/2} |J|^{1/2} |K|^{1/2} } { |Q|^2} \Big(\frac{ \ell(K) }{\ell(Q)} \Big)^{\alpha-\gamma(2n+\alpha)} \\
& = \frac{|I|^{1/2} |J|^{1/2} |K|^{1/2} } { |Q|^2} \Big(\frac{ \ell(K) }{\ell(Q)} \Big)^{\alpha/2}.
\end{align*}
This establishes the desired normalisation, and therefore we are done with $\sigma^1_1$.

\subsection*{Step II: diagonal}
Here we look at the sum
\begin{align*}
\sigma^1_2 &:= 
 \sum_{K \in \calD_{\textup{good}}} \mathop{\sum_{\substack{I,J \in \calD \colon \\ \ell(K)\le \ell (I)=2\ell(J) \\ \max( d(K,I), d(K,J)) \le \ell(K)^\gamma \ell(J)^{1-\gamma}}}}_{K \cap I = \emptyset \textup{ or } K = I \textup{ or } K \cap J = \emptyset}
\bla T(\Delta_I f ,1_J \bla g \bra_J ), \Delta_K h \bra \\
&= \sum_{K \in \calD_{\textup{good}}} \mathop{\sum_{\substack{I,J \in \calD \colon \\ \ell(K)\le \ell (I)=2\ell(J) \le 2^r\ell(K) \\ \max( d(K,I), d(K,J)) \le \ell(K)^\gamma \ell(J)^{1-\gamma}}}}_{K \cap I = \emptyset \textup{ or } K = I \textup{ or } K \cap J = \emptyset}
\bla T(h_I, h_J^0), h_K \bra \bla f, h_I \bra \bla g, h_J^0\bra \bla h, h_K \bra.
\end{align*}
The goodness of the cube $K$ was used to conclude that we cannot have $\ell(I) > 2^r\ell(K)$. Indeed, in the case $K \cap I = \emptyset$ this would imply
$d(K,I) > \ell(K)^\gamma \ell(I)^{1-\gamma} \ge \ell(K)^\gamma \ell(J)^{1-\gamma}$ -- a contradiction. In the case $K \cap J = \emptyset$ we would have (as $\ell(J) \ge 2^r\ell(K)$)
that $d(K,J) > \ell(K)^\gamma \ell(J)^{1-\gamma}$ -- a contradiction.
\begin{lem}
For $I, J, K$ as in $\sigma^1_2$ there exists a cube $Q \in \calD$ so that $I \cup J \cup K \subset Q$ and
$\ell(Q) \le 2^r\ell(K)$.
\end{lem}
\begin{proof}
Define $Q = K^{(r)}$. Then $\ell(Q) = 2^r\ell(K) \ge \ell(I) > \ell(J)$. Therefore, it suffices to show that $I \cap Q \ne \emptyset$ and $J \cap Q \ne \emptyset$.
But this is essentially the same argument as previously:
If we would have that $I \subset Q^c$ or $J \subset Q^c$, we would get
$$
\ell(K)^{\gamma}\ell(Q)^{1-\gamma} < d(K, Q^c) \le \max(d(K,I), d(K,J)) \le \ell(K)^{\gamma}\ell(J)^{1-\gamma},
$$
which implies $\ell(J) > \ell(Q)$ -- a contradiction.
\end{proof}
We can now write
$$
\sigma^1_2= \sum_{k=0}^r \sum_{i=0}^k \sum_Q \sum_{\substack{I,J \in \calD,\, K \in \calD_{\good} \colon \\ \max( d(K,I), d(K,J)) \le \ell(K)^\gamma \ell(J)^{1-\gamma} \\ K \cap I = \emptyset \textup{ or } K = I \textup{ or } K \cap J = \emptyset
\\ 2\ell(J) = \ell(I) = 2^{-i}\ell(Q), \, \ell(K) = 2^{-k}\ell(Q) \\  I \vee J \vee K = Q}} \bla T(h_I, h_J^0), h_K \bra \bla f, h_I \bra \bla g, h_J^0\bra \bla h, h_K \bra.
$$
Notice that if $K \cap I = \emptyset$ then
\begin{align*}
|\bla T(h_I, h_J^0), h_K \bra| &\lesssim |I|^{-1/2} |J|^{-1/2} |K|^{-1/2} \int_{10I \setminus I} \int_I \frac{dy\,dx}{|x-y|^n} \\
&\lesssim |I|^{1/2}  |J|^{-1/2} |K|^{-1/2} \sim \frac{|I|^{1/2} |J|^{1/2} |K|^{1/2} } { |Q|^2} \Big(\frac{ \ell(K) }{\ell(Q)} \Big)^{\alpha/2}.
\end{align*}
We get the same bound also if $K \cap J = \emptyset$ with an analogous calculation. So we only need to estimate in the case
$K = I$ and $J \in \textup{ch}(K)$. Then we have
\begin{align*}
|\bla T(h_K, h_J^0), h_K \bra| \lesssim |K|^{-3/2} \sum_{K', K'' \in \textup{ch}(K)}  |\bla T(1_{K'}, 1_J), 1_{K''} \bra|.
\end{align*}
If $K' \ne J$ or $K'' \ne J$ then $|\bla T(1_{K'}, 1_J), 1_{K''} \bra| \lesssim |K|$ simply by the size estimate of the kernel. In the case $K' = K'' = J$ we have using the weak boundedness property
that $|\bla T(1_{J}, 1_J), 1_{J} \bra| \lesssim |K|$.
So in the case $K = I$ and $J \in \textup{ch}(K)$ we also have
$$
|\bla T(h_K, h_J^0), h_K \bra| \lesssim |K|^{-1/2} \sim \frac{|I|^{1/2} |J|^{1/2} |K|^{1/2} } { |Q|^2} \Big(\frac{ \ell(K) }{\ell(Q)} \Big)^{\alpha/2}.
$$

The above lets us write
$$
\sigma^1_2= C \sum_{k=0}^r \sum_{i=0}^k 2^{-k\alpha/2} \bla S^{i,i+1,k}(f,g), h \bra 
$$
for cancellative bilinear shifts $S^{i,i+1,k}$, where $C$ depends on the kernel estimates and the weak boundedness property.
We point out at this point that since $T=T^\varphi_{\vare_1, \vare_2}=T^\varphi_{\vare_1}-T^\varphi_{\vare_2}$, there holds 
$$
\| T \|_{\operatorname{WBP}} \le C'\big( \| K \|_{\operatorname{CZ}_\alpha} + \sup_{\delta>0} \| T_\delta \|_{\operatorname{WBP}}).
$$

\subsection{Step III: error terms}
Here we start working with the sum
\begin{align*}
\sigma^1_3 &:= \sum_{\substack{I,J \in \calD,\, K \in \calD_{\good} \\  \ell(I) = 2\ell(J) \\ K \subset J \subset I}}
\bla T(\Delta_I f ,1_J \bla g \bra_J ), \Delta_K h \bra \\
&= \sum_{\substack{J \in \calD,\, K \in \calD_{\good} \\ K \subset J}}
\bla T(\Delta_{J^{(1)}} f ,1_J), \Delta_K h \bra \bla g \bra_J.
\end{align*}
We split
\begin{align*}
\bla T(\Delta_{J^{(1)}} f ,1_J), \Delta_K h \bra &=  \bla T( 1_{J^c}(\Delta_{J^{(1)}}f-\bla \Delta_{J^{(1)}}f \bra_{J}) , 1_J), \Delta_K h  \bra \\
 &- \bla \Delta_{J^{(1)}}f \bra_{J}\bla T(1, 1_{J^c} ), \Delta_K h  \bra
 +  \bla \Delta_{J^{(1)}}f \bra_{J} \bla T(1, 1), \Delta_K h  \bra.
\end{align*}
This gives us the decomposition $\sigma^1_3 = \sigma^1_{3, e} + \sigma^1_{3, \pi}$, where
the first two terms of the above decomposition are part of $\sigma^1_{3,e}$.

In this section we only deal with the error term $\sigma^1_{3,e}$. Notice that
\begin{align*}
\sigma^1_{3,e} = \sum_{\substack{J \in \calD,\, K \in \calD_{\good} \\ K \subset J}} &|J|^{-1/2}\big[\bla T(s_J,1_J), h_K\bra \\
&- \bla h_{J^{(1)}} \bra_{J}\bla T(1, 1_{J^c} ), h_K \bra \big]  \bla f, h_{J^{(1)}} \bra \bla g, h_J^0 \bra \bla h, h_K \bra,
\end{align*}
where $s_J := 1_{J^c}(h_{J^{(1)}}-\bla h_{J^{(1)}} \bra_{J})$ satisfies $|s_J| \lesssim |J|^{-1/2}$ and spt$\,s_J \subset J^c$.

We will first bound $|\bla T(s_J,1_J), h_K \bra|$. In the case $\ell(J) \sim \ell(K)$
we are looking for the bound $|\bla T(s_J,1_J), h_K \bra| \lesssim 1$. This follows by writing
$$
|\bla T(s_J,1_J), h_K \bra| \le |\bla T(1_{3J} s_J,1_J), h_K \bra|  + |\bla T(1_{(3J)^c} s_J,1_J), h_K \bra|,
$$
and using the size and H\"older estimate in the $x$-variable respectively. If $\ell(J) \ge 2^r \ell(K)$ we have
$d(K, J^c) \ge \ell(K)^{\gamma}\ell(J)^{1-\gamma} \ge \ell(K)^{1/2}\ell(J)^{1/2}$. Therefore, H\"older estimate in the $x$-variable gives
\begin{align*}
|\bla T(s_J,1_J), h_K \bra| &\lesssim |K|^{-1/2} |J|^{-1/2}\ell(K)^{\alpha} \int_K \int_{J^c} \frac{dy}{|x-y|^{n+\alpha}} \,dx \\
&\lesssim |K|^{1/2} |J|^{-1/2} \Big( \frac{\ell(K)}{\ell(J)} \Big)^{\alpha/2}.
\end{align*}
Notice that this is $\sim 1$ if $\ell(J) \sim \ell(K)$, so the same estimate holds in both cases.

It is now also obvious, using almost exactly the same calculations as above, that
$$
|\bla T(1, 1_{J^c} ), h_K \bra| \lesssim |K|^{1/2} \Big( \frac{\ell(K)}{\ell(J)} \Big)^{\alpha/2}.
$$
But as $|\bla h_{J^{(1)}} \bra_{J}| \lesssim |J|^{-1/2}$ we have the same bound as above. Therefore, we can write
$$
\sigma^1_{3,e}= C \sum_{k=1}^{\infty} 2^{-\alpha k / 2} \bla S^{0,1,k}(f,g), h\bra
$$
for some cancellative bilinear shifts and for some $C$ depending on the kernel estimates.

\subsection{Part IV: paraproduct}
Here we combine
$$
\sigma^1_{3, \pi} = \sum_{\substack{J \in \calD,\, K \in \calD_{\good} \\ K \subset J}} \bla T(1, 1), \Delta_K h  \bra \bla \Delta_{J^{(1)}}f \bra_{J}  \bla g \bra_J
$$
with the relevant paraproduct type term coming from $\sigma^2$, namely
$$
\sigma^2_{3, \pi} = \sum_{\substack{J \in \calD,\, K \in \calD_{\good} \\ K \subset J}} \bla T(1, 1), \Delta_K h  \bra \bla f \bra_{J^{(1)}} \bla \Delta_{J^{(1)}}g \bra_{J}.
$$
Notice the key cancellation
$$
\bla \Delta_{J^{(1)}}f \bra_{J}  \bla g \bra_J + \bla f \bra_{J^{(1)}} \bla \Delta_{J^{(1)}}g \bra_{J} = \bla f \bra_J \bla g \bra_J - \bla f \bra_{J^{(1)}}  \bla g \bra_{J^{(1)}}.
$$
Therefore, we get
$$
\sigma^1_{3, \pi} + \sigma^2_{3, \pi} = \sum_{K \in \calD_{\good}}  \bla T(1, 1), \Delta_K h  \bra \bla f \bra_K \bla g \bra_K.
$$
Define
\begin{equation}\label{eq:CoefPara0}
\alpha_K = \frac{\bla T(1,1), h_K \bra}{C \big(\| K \|_{\operatorname{CZ}_\alpha}+ \sup_{\delta>0}\| T_\delta(1,1)\|_{\BMO}\big)}
\end{equation}
if $K$ is good, where $C$ is a large enough absolute constant, and otherwise set $\alpha_K=0$. 
Recall that $T=T^\varphi_{\vare_1, \vare_2}=T^\varphi_{\vare_1}-T^\varphi_{\vare_2}$, whence in view of Lemma \ref{lem:BMOTransfer} and Lemma \ref{lem:JN}
the numbers $\alpha_K$ satisfy the correct normalisation \eqref{eq:BMO2}.
Hence we can write
$$
\sigma^1_{3, \pi} + \sigma^2_{3, \pi} = C \big(\| K \|_{\operatorname{CZ}_\alpha}+ \sup_{\delta>0}\| T_\delta(1,1)\|_{\BMO}\big)\langle \Pi_{\alpha} (f,g),h \rangle.
$$

\subsection{Synthesis}\label{sec:Synth}
Let us collect the pieces of the above steps together. Recall that the operator $T$ is actually $T^\varphi_{\vare_1, \vare_2}$. We have shown that 
\begin{equation*}
\begin{split}
\Sigma^1(\omega) 
&=C(\|K\|_{\operatorname{CZ}_{\alpha}} + \sup_{\delta>0}\|T_{\delta}\|_{\operatorname{WBP}})   
\sum_{k=0}^\infty \sum_{i=0}^{k}2^{-\alpha k/2} \bla U^{i,k}_{\varepsilon_1,\vare_2, \varphi, \omega} (f,g),h\bra \\
&+C(\|K\|_{\operatorname{CZ}_{\alpha}}+ \sup_{\delta>0}\|T_{\delta}(1,1)\|_{\operatorname{BMO}}) \bla \Pi_{\alpha_0(\varepsilon_1, \vare_2, \varphi, \omega)}(f,g), h \bra,
\end{split}
\end{equation*}
where each $U^{i,k}_{\varepsilon_1, \vare_2, \varphi, \omega}$
is a sum of  cancellative shifts $S^{i,i,k}_{\varepsilon_1,\vare_2, \varphi, \omega}$ and $S^{i,i+1,k}_{\varepsilon_1,\vare_2, \varphi, \omega}$, and where
$\Pi_{\alpha_0(\varepsilon_1, \vare_2, \varphi, \omega)}$
is the paraproduct related to the sequence defined around Equation \eqref{eq:CoefPara0}. Collecting together the symmetric parts we get the result of Theorem \ref{thm:main} except
we have the dependence on $\epsilon_2$ on both sides. However, it is clear that $\bla T^\varphi_{\vare_1, \vare_2}(f,g), h \bra = \bla T^\varphi_{\vare_1}(f,g), h \bra$ if
$\epsilon_2$ is large enough (depending on the supports of $f,g$ and $h$.) Thus, it is enough to do some limiting argument $\epsilon_2 \to \infty$ on the right hand side also.

The operators $U^{i,k}_{\varepsilon_1,\vare_2, \varphi, \omega}$ depend on $\vare_1$,  $\vare_2$ and $\varphi$ because the coefficients of the shifts are defined using the operator
$T^\varphi_{\vare_1, \vare_2}$. Let $U^{i,k}_{\varepsilon_1, \varphi, \omega}$ be the corresponding operator, but where the coefficients of the shifts are defined with the operator
$T^\varphi_{\vare_1}$ instead. Do the similar thing with the paraproducts. Dominated convergence theorem shows that it is enough to show that
$$
\bla U^{i,k}_{\varepsilon_1,\vare_2, \varphi, \omega} (f,g),h\bra \to \bla U^{i,k}_{\varepsilon_1, \varphi, \omega} (f,g),h\bra,
$$
when $\epsilon_2 \to \infty$, and similarly for the paraproducts. The convergence of the above pairings is simply based on the fact that the coefficients of the shifts defined with $T^\varphi_{\vare_1, \vare_2}$ approach to the ones
defined with $T^\varphi_{\vare_1}$. Let us quickly show the argument for the paraproduct, the same reasoning applies for the cancellative shifts.

It is enough to show that
$$
\lim_{\epsilon_2 \to \infty} \Big| \sum_{K \in \calD_{\good}} 
\bla T^\varphi_{\vare_2}(1,1), h_K \bra  \bla f \bra_K \bla g \bra_K \bla h,h_K \bra \Big| = 0.
$$
Fix $M > 0$. Notice that using $\sup_{\delta > 0 } \|T^\varphi_{\delta}(1,1)\|_{\BMO} < \infty$ and the boundness of the paraproduct there holds for every $\epsilon_2 > 0$ that
\begin{align*}
\Big|\sum_{K \colon \ell(K) < 1/M \textup{or } \ell(K)> M }
&\bla T^\varphi_{\vare_2}(1,1), h_K \bra
\bla f \bra_K \bla g \bra_K \bla h, h_K \bra\Big| \\ &\lesssim \|f\|_{L^4} \|g\|_{L^4} \Big( \sum_{K \colon \ell(K) < 1/M \textup{or } \ell(K)> M} \|\Delta_K h\|_{L^2}^2 \Big)^{1/2} = c(M),
\end{align*}
where $c(M) \to 0$ when $M \to \infty$.
This gives that
\begin{align*}
\Big|\sum_{K}
\bla T^\varphi_{\vare_2}(1,1),& h_K \bra  \bla f \bra_K \bla g \bra_K \bla h,h_K \bra \Big| \\
&\le c(M) + \Big|\sum_{K\colon 1/M \le \ell(K) \le M}
\bla T^\varphi_{\vare_2}(1,1), h_K \bra  \bla f \bra_K \bla g \bra_K \bla h,h_K \bra \Big|.
\end{align*}
The latter sum is finite as $h$ has compact support. Since $\bla T^\varphi_{\vare_2}(1,1), h_K \bra \to 0 $ when $\epsilon_2 \to \infty$,
we have that
$$
\lim_{\epsilon_2 \to \infty} \Big| \sum_{K \in \calD_{\good}} 
\bla T^\varphi_{\vare_2}(1,1), h_K \bra  \bla f \bra_K \bla g \bra_K \bla h,h_K \bra \Big| \le c(M).
$$
The claim follows by letting $M \to \infty$.

We are done with the proof of Theorem \ref{thm:main}.

\section{Sparse form domination for shifts}\label{sec:SFPARSE}
Let us first introduce a general framework of trilinear forms. 
Let $\mathcal{D}$ be a fixed dyadic grid on $\mathbb{R}^n$ and $i,j,k$ be nonnegative integers. Define the trilinear form 
\begin{align*}
\mathbb{S}^\rho(f_1,f_2,f_3) &:=\sum_{Q\in\mathcal{D}}S_Q(f_1,f_2,f_3) \\ 
&:=\sum_{Q\in\mathcal{D}}\iiint_{Q\times Q\times Q} K_Q(x_1,x_2,x_3)\prod_{j=1}^3f_j(x_j)\,dx_1\, dx_2\, dx_3,
\end{align*}
where $\rho \ge 0$. Assume it satisfies the following:
\begin{itemize}

\item[A.] The kernels $K_Q:\,Q\times Q\times Q\to \mathbb{C}$ satisfy $\|K_Q\|_{L^\infty}\leq |Q|^{-2}$.

\item[B.] There exist exponents $p,q,r \in (1,\infty)$ such that $1/p+1/q=1/r$ and a constant $\mathcal{B}$ so that 
for every subcollection $\mathcal{Q} \subset \calD$ of dyadic cubes the truncated form
\[
\mathbb{S}^\rho_{\mathcal{Q}}(f_1,f_2,f_3):=\sum_{Q\in\mathcal{Q}}S_Q(f_1,f_2,f_3).
\]
satisfies
$$
|\mathbb{S}^\rho_{\mathcal{Q}}(f_1,f_2,f_3)|
\le \mathcal{B} \| f_1\|_{L^p} \| f_2 \|_{L^q} \|f_3 \|_{L^{r'}}.
$$

\item[C.] $K_Q$ is constant on sets of the form $Q_1 \times Q_2 \times Q_3$, where $Q_i^{(\rho+1)} = Q$.
\end{itemize}
It can easily be seen that trilinear forms associated to both cancellative bilinear shifts and paraproducts fall into the above class of forms. Corollary \ref{cor:main} follows
from Thereom \ref{thm:main} by using two results from this section, namely Proposition \ref{prop:SparseInFixedGridSF} and Corollary \ref{cor:UniversalSparse}.

We state the next proposition for only dyadic grids without quadrants --  these are dyadic
grids where every sequence of cubes $I_k$ with $I_k \subsetneq I_{k+1}$ satisfy $\R^n = \bigcup_k I_k$. Since almost every 
dyadic grid has this property, this generality is already enough for us to conclude everything we need. Of course, the proposition would hold in every
grid but since this is not needed, we prefer this technical simplification. 
\begin{prop}\label{prop:SparseInFixedGridSF}
Let $\eta \in (0,1)$, $\calD$ be a dyadic grid without quadrants and $f_1, f_2, f_3$ be compactly supported and bounded functions. 
Then there exists an $\eta$-sparse collection $\mathcal{S}=\mathcal{S}(f_1,f_2,f_3, \eta)\subset \calD$, 
so that for all $\mathbb{S}^\rho$ defined in $\calD$ there holds
\begin{equation}\label{sparseforshifts}
\mathbb{S}^\rho(f_1,f_2,f_3)\lesssim_{\eta} ( \mathcal{B}+\rho) \sum_{Q\in\mathcal{S}}|Q|\prod_{j=1}^3\bla |f_j|\bra_Q=:(\mathcal{B}+\rho)\Lambda_{\mathcal{S}}(f_1,f_2,f_3).
\end{equation}
\end{prop}
\begin{proof}
Let $Q_0 \in \calD$ be so that
it contains the supports of all of the three functions $f_j$.
Define $\mathcal{E}$ to be the collection of maximal cubes $Q \in \calD$, $Q \subset Q_0$, such that 
$$
\max \left(\frac{\bla |f_1|\bra_{Q}}{\bla |f_1|\bra_{Q_0}},\,  
\frac{\bla |f_2|\bra_{Q}}{\bla |f_2|\bra_{Q_0}},\,
\frac{\bla |f_3|\bra_{Q}}{\bla |f_3|\bra_{Q_0}}\right)> C_0.
$$
For $C_0 = C_0(\eta)$ large enough there holds
\[
\sum_{Q\in\mathcal{E}}|Q| \le (1-\eta)|Q_0|.
\]
The cube $Q_0$ is the first cube to be included in $\mathcal{S}$, and $E_{Q_0} := Q_0 \setminus \bigcup_{Q \in \mathcal{E}} Q$.

Let $\mathcal{G} = \mathcal{G}(Q_0) :=\{Q \in \calD \colon Q \subset Q_0 \text{ and } Q \not \subset Q' \text{ for every } Q' \in \mathcal{E}\}$,
and for $Q \in \calD$ write $\calD(Q)=\{R \in \calD \colon R \subset Q\}$. Then we have the decomposition 
\begin{equation}\label{eq:FirstSplit}
\begin{split}
\mathbb{S}^\rho(f_1 ,f_2 ,f_3 )
&= \sum_{\substack{Q \in \calD \\ Q \supsetneq Q_0}} S_Q (f_1,f_2,f_3)
+\mathbb{S}_{\mathcal{G}}^\rho(f_1 ,f_2 ,f_3) \\
&+\sum_{Q \in \mathcal{E}} \mathbb{S}_{\calD(Q)}^\rho(f_1 1_{Q},f_2 1_{Q},f_3 1_{Q}),
\end{split}
\end{equation}
where we applied the fact that the functions are supported in $Q_0$. 
The size property $\|K_Q \|_{L^\infty} \le |Q|^{-2}$ of the kernels implies that
\begin{equation*}
\begin{split}
\Big| \sum_{\substack{Q \in \calD \\ Q \supsetneq Q_0}} S_Q (f_1,f_2,f_3) \Big|
\le \sum_{\substack{Q \in \calD \\ Q \supsetneq Q_0}}
\frac{\| f_1 \|_{L^1}\| f_2 \|_{L^1}\| f_3 \|_{L^1}}{|Q|^2}
\sim |Q_0|\prod_{j}\bla |f_j|\bra_{Q_0}.
\end{split}
\end{equation*}

We will prove the estimate
\begin{equation}\label{eq:EstIter}
\mathbb{S}_{\mathcal{G}}^\rho(f_1 ,f_2 ,f_3 )
\lesssim_{\eta} (\mathcal{B}+\rho) |Q_0|\prod_{j}\bla |f_j|\bra_{Q_0}.
\end{equation}
From \eqref{eq:FirstSplit} and \eqref{eq:EstIter} it is then seen that the collection
$\mathcal{S}$ can be obtained by iterating this process, in the second step beginning with 
$ \mathbb{S}_{\calD(Q)}^\rho(f_1 1_{Q},f_2 1_{Q},f_3 1_{Q})$ for some $Q \in \mathcal{E}$.
Hence, to conclude the proof, it remains to show \eqref{eq:EstIter}.

We prove \eqref{eq:EstIter} by performing a Calder\'on-Zygmund decomposition to $f_j$ with respect to the collection $\mathcal{E}$, obtaining for each $j=1,2,3$ that
\[
f_j=g_j+b_j:=g_j+\sum_{Q\in\mathcal{E}}b_{j,Q},\quad b_{j,Q}:=\left(f_j-\bla f_j\bra_Q\right)1_Q.
\]
For every $Q \in \mathcal{E}$ there hold the standard properties
\[
\|g_j\|_{L^\infty}\lesssim_{\eta} \bla |f_j|\bra_{Q_0},
\quad \int_Q b_{j,Q}=0, 
\quad \|b_{j,Q}\|_{L^1}\lesssim_{\eta} |Q|\bla |f_j|\bra_{Q_0}.
\]

Decompose the left hand side of \eqref{eq:EstIter} into eight parts:
\[
\mathbb{S}^\rho_{\mathcal{G}}(g_1,g_2,g_3),\,\mathbb{S}^\rho_{\mathcal{G}}(b_1,b_2,b_3),\,\mathbb{S}^\rho_{\mathcal{G}}(g_1,g_2,b_3),\,\mathbb{S}^\rho_{\mathcal{G}}(g_1,b_2,g_3),\,\cdots
\]
The part with three good functions can be directly estimated via the boundedness of $\mathbb{S}^\rho_{\mathcal{G}}$ and the estimates $\|g_j\|_{L^\infty}\lesssim_{\eta} \bla |f_j|\bra_{Q_0}$:
$$
|\mathbb{S}^\rho_{\mathcal{G}}(g_1,g_2,g_3)| \le \mathcal{B} \|g_1\|_{L^p} \|g_2\|_{L^q}  \|g_3\|_{L^{r'}} \lesssim_{\eta} \mathcal{B} |Q_0| \prod_j \bla |f_j|\bra_{Q_0}.
$$

In all the other parts, there is at least one bad function involved. All of these terms vanish by assumption C if $\rho = 0$, so assume now that $\rho \ge 1$.
By symmetry we consider a term of the form $\mathbb{S}_{\mathcal{G}}^\rho(b_1,h_2,h_3)$, 
where $h_j$ can either be $g_j$ or $b_j$. 
We further decompose $\mathcal{G}$ into $\rho$ subcollections each of which, 
denoted by $\mathcal{G}'$, satisfies that $\ell(I_1)\ge 2^{\rho}\ell(I_2)$ 
whenever $I_1,I_2\in\mathcal{G}'$, $I_1\supsetneq I_2$. It suffices to show that
\begin{equation}\label{sparsereduce}
|\mathbb{S}_{\mathcal{G}'}^\rho(b_1,h_2,h_3)|\lesssim_{\eta} |Q_0|\prod_{j}\bla |f_j|\bra_{Q_0}.
\end{equation}

Because of the assumption C, the defining property of $\mathcal{G}'$ and the fact that $\int b_{1,Q} = 0$ for every $Q \in \mathcal{E}$, we have that
for every $Q\in\mathcal{E}$ there exists at most one $R\in\mathcal{G}'$ such that $Q\subsetneq R$ and $S_R(b_{1,Q},h_2,h_3)\neq 0$.
If such a cube $R$ exists we denote it by $R(Q)$. Therefore, 
\begin{equation}\label{bhh}
\begin{split}
|\mathbb{S}_{\mathcal{G}'}^\rho(b_1,h_2,h_3)|
&\leq \sum_{R\in\mathcal{G}'} \mathop{\sum_{Q\in\mathcal{E}}}_{R(Q)=R} |S_R(b_{1,Q},h_2,h_3)| \\
& \le \sum_{R\in\mathcal{G}'}\mathop{\sum_{Q\in\mathcal{E}}}_{R(Q)=R} 
\frac{\|b_{1,Q}\|_{L^1}\|h_2 1_R\|_{L^1}\|h_3 1_R\|_{L^1}}{|R|^2},
\end{split}
\end{equation}
where the size estimate $\|K_R\|_{L^\infty}\leq |R|^{-2}$ was applied.

Let $j =2,3$ and fix some $R \in \mathcal{G}'$ for the moment. 
We will prove $\|h_j 1_R\|_{L^1} \lesssim_{\eta}  |R| \bla |f_j|\bra_{Q_0}$.
The $L^\infty$ property of $g_j$ implies that
$
\|g_j 1_R\|_{L^1} \lesssim_{\eta}  |R| \bla |f_j|\bra_{Q_0}.
$
The estimates $\|b_{j,Q}\|_{L^1}\lesssim_{\eta} |Q|\bla |f_j|\bra_{Q_0}$ give
$$
\|b_j 1_R\|_{L^1}
=  \sum_{Q \in \mathcal{E} \colon Q \subset R} \| b_{j,Q}\|_{L^1}
\lesssim_{\eta} \sum_{Q \in \mathcal{E} \colon Q \subset R} |Q|\bla |f_j|\bra_{Q_0}
\le  |R| \bla |f_j|\bra_{Q_0}. 
$$

Now we proceed from \eqref{bhh} as
\begin{equation*}
\begin{split}
|\mathbb{S}_{\mathcal{G}'}^\rho(b_1,h_2,h_3)|
&\lesssim_{\eta} \sum_{R\in\mathcal{G}'}\mathop{\sum_{Q\in\mathcal{E}}}_{R(Q)=R}
 \|b_{1,Q}\|_{L^1} \bla |f_2|\bra_{Q_0}\bla |f_3|\bra_{Q_0} \\
& \lesssim_{\eta}  |Q_0|\bla |f_1|\bra_{Q_0} \bla |f_2|\bra_{Q_0}\bla |f_3|\bra_{Q_0}.
\end{split}
\end{equation*}
This completes the proof of \eqref{sparsereduce}, and hence the proof of the proposition.
\end{proof}

For clarity we give the proof of the following lemma -- it is a simple argument that can be extracted from the proof of Lemma 4.7 in Lacey--Mena \cite{LM}.
\begin{lem}\label{lem:UniversalSparseInAGrid}
Let $0 < \eta_1, \eta_2 < \infty$. Suppose $\calD$ is a dyadic grid and $f_1, f_2, f_3 \in L^1$. 
Then there is an $\eta_2$-sparse family $\mathcal{U} = \mathcal{U}(f_1, f_2, f_3,\eta_2) \subset \calD$ so that
for all $\eta_1$-sparse $\calS \subset \calD$ there holds that
$$
\Lambda_{\calS}(f_1, f_2, f_3) \lesssim_{\eta_1, \eta_2} \Lambda_{\mathcal{U}}(f_1, f_2, f_3).
$$
\end{lem}
\begin{proof}
We first construct the family $\mathcal{U}$. Let $C = C(\eta_2) \ge 8^n$ be a large enough constant depending on $\eta_2$. For each $k \in \Z$ define
$$
\mathcal{U}_k = \Big\{\textup{maximal cubes } Q \in \calD \textup{ so that } \prod_j \bla |f_j| \bra_Q > C^k\Big\}.
$$
Notice that if $Q \in \mathcal{U}_k$ then 
$$
C^k < \prod_j \bla |f_j| \bra_Q \le 8^n \prod_j \bla |f_j| \bra_{Q^{(1)}} \le 8^n C^k \le C^{k+1}.
$$
This means that a given $Q \in \calD$ can belong to at most one of the collections $\mathcal{U}_k$. Define
$$
\mathcal{U} = \bigcup_{k \in \Z} \mathcal{U}_k.
$$
Let us show that this is an $\eta_2$-sparse collection. Let $Q \in \mathcal{U}$ and fix $k$ so that $Q \in \mathcal{U}_k$. 
Notice first that
$$
\Big| \mathop{\bigcup_{R \in \mathcal{U}}}_{R \subsetneq Q} R\Big| = \Big| \mathop{\bigcup_{R \in \mathcal{U}_{k+1}}}_{R \subset Q} R\Big|
= \mathop{\sum_{R \in \mathcal{U}_{k+1}}}_{R \subset Q} |R|.
$$
If $R \in \mathcal{U}_{k+1}$ is such that $R \subset Q$, then
$$
\prod_j \bla |f_j| \bra_R > C^{k+1} \ge \frac{C}{8^n} \prod_j \bla |f_j| \bra_Q,
$$
and so
$$
\max_j \frac{\bla |f_j| \bra_R}{\bla |f_j| \bra_Q} > \frac{C^{1/3}}{2^n}.
$$
This implies that
$$
\Big| \mathop{\bigcup_{R \in \mathcal{U}}}_{R \subsetneq Q} R\Big|  \le \frac{3\cdot 2^n}{C^{1/3}} |Q| \le (1-\eta_2)|Q|
$$
provided $C = C(\eta_2)$ is large enough. It is now clear that the sets
$$
E_Q := Q \setminus \mathop{\bigcup_{R \in \mathcal{U}}}_{R \subsetneq Q} R, \qquad Q \in \mathcal{U},
$$
are disjoint and satisfy $|E_Q| \ge \eta_2 |Q|$, which proves that $\mathcal{U}$ is $\eta_2$-sparse.

Consider an arbitrary $\calS \subset \calD$, which is $\eta_1$-sparse. If $Q \in \calS$ satisfies
$\prod_j \bla |f_j| \bra_Q \ne 0$, then there is a cube $R \in \mathcal{U}$ so that $Q \subset R$. Let
$\pi_{\mathcal{U}} Q$ denote the minimal $R \in \mathcal{U}$ so that $Q \subset R$. Suppose $\pi_{\mathcal{U}} Q \in \mathcal{U}_k$.
Then we cannot have $\prod_j \bla |f_j| \bra_Q > C^{k+1}$ (as otherwise $\pi_{\mathcal{U}} Q$ would not be minimal), and so
$$
\prod_j \bla |f_j| \bra_Q \le C^{k+1} \le C \prod_j \bla |f_j| \bra_{\pi_{\mathcal{U}} Q} \lesssim_{\eta_2} \prod_j \bla |f_j| \bra_{\pi_{\mathcal{U}} Q}.
$$
Finally, we get
\begin{align*}
\Lambda_{\calS}(f_1, f_2, f_3) & = \sum_{R \in \mathcal{U}} \mathop{\sum_{Q \in \calS}}_{\pi_{\mathcal{U}} Q = R} |Q| \prod_j \bla |f_j| \bra_Q \\
& \lesssim_{\eta_2} \sum_{R \in \mathcal{U}} \prod_j \bla |f_j| \bra_{R} \mathop{\sum_{Q \in \calS}}_{Q^a = R} |Q| \\
& \lesssim_{\eta_1} \sum_{R \in \mathcal{U}} |R| \prod_j \bla |f_j| \bra_{R} = \Lambda_{\mathcal{U}}(f_1, f_2, f_3).
\end{align*}
\end{proof}

\begin{cor}\label{cor:UniversalSparse}
There exists dyadic grids $\mathcal{D}_i$, $i = 1, \ldots, 3^n$, with the following property. Let $\eta_1, \eta_2 \in (0,1)$.
Suppose $f_1,f_2,f_3 \in L^1$.
Then for some $i$ there exists an $\eta_2$-sparse collection $\mathcal{U}=\mathcal{U}(f_1,f_2,f_3, \eta_2)\subset \calD_i$, 
so that for all $\eta_1$-sparse collections of cubes $\mathcal{S}$ we have
\begin{equation}\label{universalsparse}
\Lambda_{\mathcal{S}}(f_1, f_2, f_3) \lesssim_{\eta_1, \eta_2} \Lambda_{\mathcal{U}}(f_1, f_2, f_3).
\end{equation}
\end{cor}
\begin{proof}
We can let $(\calD_i)_i$ be any collection of $3^n$ dyadic grids with the property that for any cube $P \subset \R^n$ there exists
$R \in \bigcup_i \calD_i$ so that $P \subset R$ and $\ell(R) \le 6\ell(P)$. Then it is easy to find a $6^{-n}\eta_1$-sparse collections $\calS_i \subset \calD_i$ (depending on $\calS$)
so that
$$
\Lambda_{\mathcal{S}}(f_1,f_2,f_3) \lesssim_{\eta_1} \sum_i \Lambda_{\mathcal{S}_i}(f_1,f_2,f_3).
$$
Let $\mathcal{U}_i = \mathcal{U}_i(f_1, f_2, f_3, \eta_2) \subset \calD_i$ be the universal sparse collections given by Lemma \ref{lem:UniversalSparseInAGrid}.
Then we have that
$$
\Lambda_{\mathcal{S}}(f_1,f_2,f_3) \lesssim_{\eta_1} \sum_i \Lambda_{\mathcal{S}_i}(f_1,f_2,f_3) \lesssim_{\eta_1, \eta_2}  \sum_i \Lambda_{\mathcal{U}_i}(f_1,f_2,f_3)
\lesssim \Lambda_{\mathcal{U}_{i_0}}(f_1,f_2,f_3)
$$
for some $i_0$. We are done.
\end{proof}

\end{document}